\documentclass[centertags,12pt]{amsart}
\usepackage{amssymb}
\usepackage{verbatim}
\usepackage{array}
\usepackage{latexsym}
\usepackage{enumerate}
\usepackage{amsmath}
\usepackage{amsfonts}
\usepackage{amsthm}
\usepackage{color}
\usepackage[english]{babel}

\usepackage{tikz}
\usepackage{wasysym}
\usepackage[foot]{amsaddr}

\usepackage{mathrsfs}
\usepackage{setspace}
\usepackage{ textcomp }
\usepackage{graphicx}
\usepackage{cite}
\usepackage{comment,cancel}
\graphicspath{ {images/} }

\newtheorem{theorem}{Theorem}[section]

\newtheorem{lemma}[theorem]{Lemma}

\newtheorem{example}[theorem]{Example}
\newtheorem{rem}[theorem]{Remark}
\newtheorem{conjecture}[theorem]{Conjecture}

\def\deg{\mbox{\rm deg}}

\begin{document}

\author[P. Beelen]{Peter Beelen}
\address{Department of Applied Mathematics and Computer Science, Technical University of Denmark, 2800 Kongens Lyngby, Denmark}
	\email{pabe@dtu.dk}

\author[M. Montanucci]{Maria Montanucci}
	\email{marimo@dtu.dk}

\title{A bound for the number of points of space curves over finite fields}

\begin{abstract}
For a non-degenerate irreducible curve $C$ of degree $d$ in $\mathbb{P}^3$ over $\mathbb{F}_q$, we prove that the number $N_q(C)$ of $\mathbb{F}_q$-rational points of $C$ satisfies the inequality $N_q(C) \leq (d-2)q+1$. Our result improves the previous bound $N_q(C) \leq (d-1)q+1$ obtained by Homma in \cite{Homma} and leads to a natural conjecture generalizing Sziklai’s bound for the number of points of plane curves over finite fields.
\end{abstract}

\maketitle

\vspace{0.5cm}\noindent {\em Keywords}:
Algebraic curve, Sziklai bound, rational point, finite field

\vspace{0.2cm}\noindent{\em MSC}:
14G15, 11G20
\vspace{0.2cm}\noindent

\section{Introduction}

In the series of papers \cite{H1,H2,H3}, Homma and Kim proved that for any plane curve $C$ of degree $d$ over a finite field $\mathbb{F}_q$ without $\mathbb{F}_q$-linear components, the number $N_q(C):=|C(\mathbb{F}_q)|$ of its $\mathbb{F}_q$-rational points satisfies
\begin{equation} \label{Szi}
N_q(C)\leq (d-1)q+ 1,
\end{equation}
with up to isomorphism exactly one exception, given by the $\mathbb{F}_4$-rational curve defined by the projective plane model
$$
K: (X+Y+Z)^4+ (XY+Y Z+ZX)^2+XYZ(X+Y+Z) = 0.
$$
The bound \eqref{Szi} was originally conjectured by Sziklai \cite{PSzi}, also giving explicit examples of curves whose number of rational points achieves this bound. Note that the curve $K$ satisfies $N_4(K) = 14$, exceeding the Sziklai bound by one.

It is natural to ask if a curve $C$ contained in a higher dimensional projective space $\mathbb{P}^n$, $n \ge 3$, satisfies the Sziklai bound as well.
Recently Homma \cite{Homma} proved that indeed equation \eqref{Szi} holds for such curves as well.
However no examples of space curves whose number of rational points attains this bound were provided, leaving the possibility that the bound could be improved.

When studying a curve $C$ in projective space $\mathbb{P}^n$, it is natural to assume that the curve is non-degenerate. This means that the curve is not contained in any hyperplane of $\mathbb{P}^n$.
A degenerate curve $C$ in $\mathbb{P}^n$ can namely be seen as a curve in $\mathbb{P}^{n-1}$, by viewing the hypersurface of $\mathbb{P}^n$ containing $C$ as $\mathbb{P}^{n-1}$.
Note that a degenerate curve in $\mathbb{P}^2$ necessarily is a projective line leading to the easy observation that a plane curve is non-degenerate if and only if it is not a line.
Furthermore, trivially any curve contained in $\mathbb{P}^1$ is necessarily non-degenerate and equal to $\mathbb{P}^1$.
Hence one can say that an absolutely irreducible, non-degenerate curve $C$ in $\mathbb{P}^1$ defined over $\mathbb{F}_q$ satisfies the bound $|C(\mathbb{F}_q)| \le qd+1,$ simply because the only curve under consideration is the projective line so that $d=1$ and $|C(\mathbb{F}_q)|=q+1.$
Hence the $n=1$ case for considering absolutely irreducible, non-degenerate curves in $\mathbb{P}^n$ may not be very interesting, but it leads in a natural way to the following conjecture.
\begin{conjecture} \label{conj}
Let $\mathbb{F}_q$ be a finite field with $q$ elements. Further, let $C \subseteq \mathbb{P}^n$, $n \geq 3$ be an absolutely irreducible, non-degenerate algebraic curve of degree $d$ defined over $\mathbb{F}_q$. Then $|C(\mathbb{F}_q)| \leq q(d-n+1)+1$.
\end{conjecture}

As noted before for $n=1$ the conjecture is trivial and for $n=2$ it was shown to be true for all plane curves, except those isomorphic to the exceptional curve $K$.
Note that any non-degenerate curve in $\mathbb{P}^n$ has degree at least $n$, so that the expression $d-n+1$ is at least one.
In this paper, we prove the conjecture for $n=3$ and provide partial results for $n >3$.
For $n=3$ we also explore the consequences of the conjecture for more general, possibly reducible, curves and prove the following:
\begin{theorem} \label{main}
Let $\mathbb{F}_q$ be a finite field with $q$ elements. Further, let $C \subseteq \mathbb{P}^3$ be a non-degenerate algebraic curve of degree $d$ defined over $\mathbb{F}_q$. Moreover, suppose that if $C$ is not absolutely irreducible, then it contains no lines defined over $\mathbb{F}_q$ and that it contains at least one non-degenerate $\mathbb{F}_q$-irreducible component. Then $|C(\mathbb{F}_q)| \leq q(d-2)+1$.
\end{theorem}
For $n>3$, we prove Conjecture \ref{conj} for large values of $d$, while we also give some examples of absolutely irreducible, non-degenerate curves attaining the bound in Conjecture \ref{conj} with equality. 
%

\section{Preliminaries}

In this section, we collect several known preliminary results from the literature that we will use later on. We start with a known bound on the number of points on a non-degenerate curve.
%
%

\begin{theorem}{\rm{\cite[Theorem 3.2]{Homma}}} \label{thm:homma}
Let $C$ be a non-degenerate, irreducible curve in $\mathbb{P}^n(\mathbb{F}_q)$ of degree d.
Then $$|C(\mathbb{F}_q)| \le \dfrac{(q-1)(q^n+\cdots+1)}{q^n+\cdots+q-n}d=\dfrac{(q-1)(q^{n+1}-1)}{q(q^n-1)-n(q-1)}d.$$
\end{theorem}
Of course the upper bound stated in this theorem need not be an integer, so a trivial improvement can be obtained by taking the floor of it.

Next we recall the result from \cite{H1,H2,H3} on plane curves that we mentioned in the introduction.
\begin{theorem} \label{plane}
Let $C$ be a plane curve of degree $d$ over $\mathbb{F}_q$ without $\mathbb{F}_q$-linear components, then
$$|C(\mathbb{F}_q)| \leq q(d-1)+1,$$
unless $q=4$ and $C$ is isomorphic to the curve $K: (X+Y+Z)^4+(XY+YZ+ZX)^2+XYZ(X+Y+Z)=0$. In the latter case $|C(\mathbb{F}_4)|=14$.
\end{theorem}

While proving Conjecture \ref{conj}, we will need some results on \textit{strange} curves.
A curve $C$ is said to be strange if the tangent lines at the non-singular points of $C$ all intersect at a common point, called a nucleus. A trivial example of a strange curve is a line. A strange curve that is not a line, has in fact exactly one nucleus \cite[Theorem 1.28]{HKT}.
Non-trivial strange curves are known to exist in positive characteristic.
An example of a non-singular strange curve is given by an irreducible conic in characteristic $2$. In fact, the following result, due to Samuel \cite{Samuel}, classifies all non-singular strange curves, also see \cite[Theorem IV.3.9]{Hart} and \cite[Remark (2.2)]{BH}.
\begin{theorem} \label{strange}
The only non-singular strange curves in $\mathbb{P}^n$ are lines and irreducible conics in characteristic $2$.
\end{theorem}

The genus of a non-strange, non-degenerate curve of degree $d$ in $\mathbb{P}^n$ can be bounded from above in terms of $d$ and $n$. This bound is known as Castelnuovo's bound, see \cite[Theorem 7.111]{HKT}.
\begin{theorem}{(Castelnuovo's bound)} \label{castelnuovo}
Let $C$ be a non-degenerate, irreducible curve in $\mathbb{P}^n(\overline{\mathbb{F}}_q)$ of degree $d$ and genus $g$. If $C$ is not strange, then

$$
g \leq  \frac{1}{2} m(m-1)(n-1) + m\epsilon,
$$

where $m =\lfloor (d-1)/(n-1)\rfloor$ and $\epsilon=d-1-m(n-1)$.
\end{theorem}
A corollary of this theorem is that if a non-degenerate, irreducible curve violates the Castelnuovo bound, then it is a strange curve.

\section{Proof of the conjecture for $n=3$}

In this section we prove Conjecture \ref{conj} for $n=3$. We start by showing that the conjecture implies Theorem \ref{main}.
\begin{lemma}\label{conj_implies_thm}
If Conjecture \ref{conj} is valid for $n=3$, then Theorem \ref{main} is valid as well.
\end{lemma}
\begin{proof}
Assume that Theorem \ref{main} is proven for absolutely irreducible curves. Let $C$ be a non-degenerate curve of degree $d$ in $\mathbb{P}^3$ defined over $\mathbb{F}_q$ which is not absolutely irreducible and with at least one non-degenerate irreducible component.

Assume first that $C$ is irreducible over $\mathbb{F}_q$ but not absolutely irreducible. In this case $C$ can be written as the union of $C_1 \cup \ldots \cup C_t$ where the curves $C_i$ are absolutely irreducible defined over $\mathbb{F}_{q^t}$ and $\deg(C_i)=e \geq 1$ for all $i=1,\ldots,t$.
Note that the curves $C_1,\dots,C_t$ form an orbit under the Frobenius automorphism of $\mathbb{F}_{q^t}$ over $\mathbb{F}_q$.
Then $C(\mathbb{F}_q) \subset C_1 \cap \ldots \cap C_t$ and every point in $C(\mathbb{F}_q)$ is therefore a multiple point of $C$ of multiplicity $t \geq 2$.
Note that $q(d-2)-1\geq q-1 \geq 1$ so either $|C(\mathbb{F}_q)| \le q(d-2)-1$, or $C(\mathbb{F}_q)$ contains at least $2$ distinct points.
 Let $P_1,P_2 \in C(\mathbb{F}_q)$ be distinct and let $\ell$ be the line containing $P_1$ and $P_2$. Denote with $\mathcal{B}(\ell)$ the set of all planes $\pi$ in $\mathbb{P}^3(\mathbb{F}_q)$ containing $\ell$. Then $|\mathcal{B}(\ell)|=q+1$. Since any point in $C(\mathbb{F}_q) \cap \pi$ will have intersection multiplicity at least $t$, B{\'e}zout's theorem implies that for all $\pi \in \mathcal{B}(\ell)$:
$$|(C(\mathbb{F}_q) \setminus \{P_1,P_2\}) \cap \pi| \leq \frac{d-2t}{t}.$$
Note the we can apply B{\'e}zout's theorem. Indeed, since $\pi$ is fixed by the Frobenius map, a component $C_i$ of $C$ could only be contained in $\pi$ if all components $C_i$, and hence $C$ itself, are contained in $\pi$, contradicting the non-degeneracy of $C$.
We conclude that
$$|C(\mathbb{F}_q)| \leq 2+|\mathcal{B}(\ell)|\frac{d-2t}{t}=2+(q+1)\frac{d-2t}{t} \leq q(d-2)-1.$$

Now consider the case in which $C$ is not irreducible over $\mathbb{F}_q$. Since $C$ is line free, it can in this case be written as $C=C_1 \cup C_2 \cup \ldots \cup C_t$, $t \geq 2$ where the curves $C_i$ are $\mathbb{F}_q$-irreducible of degree $d_i \geq 2$ and defined over $\mathbb{F}_q$ for all $i=1,\ldots,t$.
Denote by $k$ the number of components $C_i$ that are degenerate. Then $k \leq t-1$, since we assume that at least one of the irreducible components is non-degenerate.
%
We can assume without loss of generality that the $k$ degenerate components of $C$ are $C_1,\ldots,C_k$. From Theorem \ref{plane} either $|C_i(\mathbb{F}_q)| \leq (d_i-1)q+1$ or $q=4$ and $|C_i(\mathbb{F}_4)|=14$ for $i=1,\ldots,k$.
%
%
If either $q \ne 4$ or $q=4$ but $|C_i(\mathbb{F}_q)| \leq (d_i-1)q+1$ for all $i=1,\ldots,k$ we get
$$|C(\mathbb{F}_q)| =\sum_{i=1}^{t} |C_i (\mathbb{F}_q)| \leq \sum_{i=1}^{t-1} [q(d_i-1)+1]+(d_{t}-2)q+1=q(d-t-1)+t \leq q(d-3)+2\le q(d-2).$$
Thus we assume that $q=4$ and $k_1 \ne 0$ is the number of degenerate components $C_i$ of $C$ with $|C_i(\mathbb{F}_4)|=14$. Relabeling the curves if necessary, we can suppose that these curves are $C_1,\ldots,C_{k_1}$. Further, since by Theorem \ref{plane} all these $C_i$ are isomorphic to the curve $K$ mentioned there, we have $d(C_i)=4$ for $1 \le i \le k_1$.
%
Since $k_1 \le k \le t-1$ and $t \ge 2$, we obtain
\begin{eqnarray}
|C(\mathbb{F}_4)| &=& 14k_1+\sum_{i=k_1+1}^{k} |C_i (\mathbb{F}_4)|+\sum_{i=k+1}^{t} |C_i (\mathbb{F}_4)| \nonumber \\
 & \leq & 14k_1+\sum_{i=k_1+1}^{t-1} [4(d_i-1)+1]+4(d_t-2)+1 \nonumber \\
&=&14k_1+4(d-4k_1)-3(t-k_1-1)-7 \nonumber \\
 & = & 4(d-2)+1+k_1-3t+3\le 4(d-2)-1. \nonumber
\end{eqnarray}
\end{proof}

\begin{rem}
The proof of Lemma \ref{conj_implies_thm} actually shows that if Conjecture \ref{conj} is valid, then equality in the bound of Theorem \ref{main} can only be obtained by absolutely irreducible curves, or possibly by reducible curves containing an $\mathbb{F}_q$-rational line.
\end{rem}

From now we assume that $C$ is absolutely irreducible and set out to prove Conjecture \ref{conj} for $n=3$. We first prove that the conjecture holds for general $n$ whenever the degree $d$ is large enough with respect to $q$ and $n$.

\begin{lemma} \label{lem1}
Let $n \ge 3$ be an integer and $C \subset \mathbb{P}^n$ a non-degenerate, absolutely irreducible curve defined over $\mathbb{F}_q$ of degree $d$. If $d \geq (n-1)q+1$, then $|C(\mathbb{F}_q)| \le q(d-n+1)$. If $d=(n-1)q$, then $|C(\mathbb{F}_q)| \le q(d-n+1)+1.$
\end{lemma}
\begin{proof}
The upper bound on $|C(\mathbb{F}_q)|$ in Theorem \ref{thm:homma} is less than or equal to
$q(d-n+1)$ if
$$d \ge (n-1)q+\dfrac{q(q-1)(n^2-1)}{q^n+\cdots+q^2-nq+1}$$
and less than or equal to $q(d-n+1)+1$ if
$$d \ge (n-1)q-1+\dfrac{((n-1)q-1)(q-1)(n+1)}{q^n+\cdots+q^2-nq+1}.$$
Using that $n \ge 3$ and $q \ge 2$ is a prime power, the lemma follows immediately, except for $n \in\{3,4\}$ and $q \le 7$. For these values of $(q,n)$ the lemma follows by strengthening the upper bound on $|C(\mathbb{F}_q)|$ in Theorem \ref{thm:homma} in a trivial way by taking the floor of it.
\end{proof}
\begin{rem}
Similarly as in the proof of Lemma \ref{lem1}, one can show that $|C(\mathbb{F}_q)| \le q(d-n+1)$ if $d=(n-1)q$ and $(q,n) \not\in \{(2,3),(2,4),(3,3),(3,4),(3,5)\}$.
\end{rem}

Lemma \ref{lem1} implies that Conjecture \ref{conj} is valid for ``large'' degrees, i.e., for $d \ge (n-1)q$. Returning to the case $n=3,$ we could assume from now on that $d \leq 2q-1$. Nonetheless, we will only assume that $d \le 2q$ in order to investigate which type of curves can satisfy the bound of Conjecture \ref{conj} with equality.

One way to estimate the number of points in $C(\mathbb{F}_q)$, is by considering lines through such points.  The following lemma shows that Conjecture \ref{conj} is true if there exists a line intersecting the curve in at least four distinct points.
 \begin{lemma} \label{lem2}
Let $C$ be an absolutely irreducible, non-degenerate curve in $\mathbb{P}^3$ defined over $\mathbb{F}_q$ of degree $d \le 2q$. If there exists an $\mathbb{F}_q$-rational line $\ell$ such that $|\ell \cap C(\overline{\mathbb{F}}_q)|\geq 4$, then $|C(\mathbb{F}_q)| \le q(d-2)$.
\end{lemma}

\begin{proof}
Let us write $m:=|\ell \cap C(\overline{\mathbb{F}}_q)|$ and denote with $\mathcal{B}(\ell)$ the set of all planes in $\mathbb{P}^3(\mathbb{F}_q)$ containing $\ell$. From B{\'e}zout's theorem we know that for an arbitrary plane $\pi \in \mathcal{B}(\ell)$,
\begin{equation}\label{lem2_eq}
|\pi \cap [C(\mathbb{F}_q) \setminus \ell]| \leq d-m.
\end{equation}
Since $|\mathcal{B}(\ell)|=q+1$, $|\ell \cap C(\mathbb{F}_q)| \le m$, $m \geq 4$, and $d \leq 2q$, we have $|C(\mathbb{F}_q)| \leq m+(q+1)(d-m) \leq q(d-2).$
\end{proof}

\begin{rem}\label{rem:singular_intersection}
The condition that $|\ell \cap C(\overline{\mathbb{F}}_q)|\geq 4$ can be weakened somewhat. What is needed to prove equation \eqref{lem2_eq}, is that the sum of the intersection multiplicities of any plane $\pi \in \mathcal{B}(\ell)$ with $C$ at the points in $\ell \cap C(\overline{\mathbb{F}}_q)$ is at least $4$. I.e., the condition $|\ell \cap C(\overline{\mathbb{F}}_q)|\ge 4$ in Lemma \ref{lem2} can be replaced by the weaker condition:
\begin{equation}\label{eq:condition4}
\forall \pi \in \mathcal{B}(\ell),  \sum _{P \in \ell \cap C(\overline{\mathbb{F}}_q)} i(\pi \cap C,P) \ge 4.
\end{equation}
The proof of Lemma \ref{lem2} is completely valid if the condition in equation \eqref{eq:condition4} is satisfied.
If for example $\ell$ is the tangent line of $C$ at $P$, then for any $\pi \in \mathcal{B}(\ell)$, we have $i(\pi \cap C,P)\ge 2$, see \cite[Page 4]{SV}. Hence if the tangent line at $P$ intersects $C$ in a further two points, or is a bitangent line of $C$, then equation \eqref{eq:condition4} and hence the conclusion of Lemma \ref{lem2} is valid.

Another case where we can use such ideas is if a point $P \in C(\overline{\mathbb{F}}_q)$ is a singularity of $C$, since then for any plane $\pi$ containing $P$, we have $i(\pi \cap C,P) \ge 2$, see \cite[Page 3]{SV}.

\end{rem}

Considering lines through one particular point $P$, leads to projection of $C$ on a plane $\pi$ in a natural way.
More precisely: given $P \in C(\mathbb{F}_q)$, let $\pi \subset \mathbb{P}^3$ be an $\mathbb{F}_q$-rational plane not containing $P$.
Then we can define the projection of a point $Q \ne P$ of $C$ on $\pi$ to be the intersection point of $\pi$ and the line connecting $P$ and $Q$.
This gives rise to a map $\rho_P: C\setminus\{P\} \rightarrow \pi$.  The projection $C^\prime$ of $C$ on $\pi$ is then defined to be the Zariski closure in $\pi$ of the image of $\rho_P$, that is $C^\prime:=\overline{\rho_P(C\setminus\{P\})}$.
If $P$ is a non-singular point, the domain of the map can be extended to include $P$ in a natural way, by defining $\rho_P(P)$ as the intersection of $\pi$ and the tangent line of $C$ at $P$. However, if $P$ is a singular point of $C$, it may not be possible to extend the domain of $\rho_P$ to $C$. For example, the singular point $P$ may have several branches with distinct tangent lines, in which case several points of $\pi$ are equally valid as image of $P$ under $\rho_P$. Taking the Zariski closure of $\rho_P(C\setminus\{P\})$ takes care of this problem and in a sense adds all possible images of $P$ under $\rho_P$ at the same time.
Either way, in general, we only have a rational map $\rho_P: C \to C^\prime$.
Generically this map would be birational, but there are cases in which its degree is strictly larger than one \cite{Ballico}.
Note that if $C$ is a non-degenerate curve, it is not a line so that the rational map $\rho_P$ is not constant.
Finally, we observe that since $C$ is absolutely irreducible, so is $C^\prime$.

We will now investigate the degree of the curve $C^\prime$. It is well known that the degree of $C^\prime$ is at most $d-1$. Indeed, a line $\lambda \subset \pi$ contains at most $d-1$ points when counted with multiplicity, since the intersection of $C$ and the plane spanned by $\lambda$ and $P$ contains $d$ points when counted with multiplicity, the extra point being $P$. We will use a more general result though, which we state and prove now.

\begin{lemma}\label{lem:projection_map}
Let $C$ be an absolutely irreducible, non-degenerate curve in $\mathbb{P}^3$ defined over $\mathbb{F}_q$ of degree $d$. Further, let $P \in C(\mathbb{F}_q)$ be a point, $\pi \subset \mathbb{P}^3$ be an $\mathbb{F}_q$-rational plane not containing $P$, and $C^\prime$ the projection of $C$ on $\pi$. Assume that the curve $C^\prime$ has degree $d^\prime$ and that the rational map $\rho_P: C \to C^\prime$ has degree $\delta$. Then
$$d \ge
\begin{cases}
1+\delta d^\prime, & \text{if $P$ is non-singular},\\
2+\delta d^\prime, & \text{if $P$ is singular}.
\end{cases}
$$
\end{lemma}
\begin{proof}
If we choose homogeneous coordinates $y_0,\dots,y_3$ for $\mathbb{P}^3$, then after applying a suitable projectivity, we may assume that $P=(0:0:0:1)$, that $\pi$ is given by the equation $y_3=0$. Then the map $\rho_P: C \rightarrow \pi$ is the natural projection $(y_0:y_1:y_2:y_3) \mapsto (y_0:y_1:y_2:0)$. Denoting by $x_i$ the coset $y_i+I \subset \mathbb{F}_q[y_0,\dots,y_3]$, where $I$ is the homogeneous ideal associated to $C$, we can describe the function fields $F$ and $F^\prime$ of $C$ and $C^\prime$ over $\overline{\mathbb{F}}_q$ as $F:=\overline{\mathbb{F}}_q(C)=\overline{\mathbb{F}}_q(x_1/x_0,x_2/x_0,x_3/x_0)$ and $F^\prime:=\overline{\mathbb{F}}_q(C^\prime)=\overline{\mathbb{F}}_q(x_1/x_0,x_2/x_0).$
Note that the degree of the rational map $\rho_P: C \to C^\prime$ equals the degree of the function field extension $\overline{\mathbb{F}}_q(C)/\overline{\mathbb{F}}_q(C^\prime)$.

Now choose $\alpha,\beta,\gamma \in \overline{\mathbb{F}}_q$, not all zero, such that
\begin{enumerate}
\item the intersection point $R$ in $\pi$ of the lines $\lambda$ given by $y_0=0$ and $\mu$ given by $\alpha y_1+\beta y_2+\gamma y_0=0$ is not a point of $C^\prime$,
\item none of the points in $C^\prime \setminus \rho_P(C \setminus\{P\})$ lies on the line $\mu$,
\item none of the singularities of $C^\prime$ lies on $\mu$,
\item the plane spanned by $\mu$ and $P$ does not contain any singularities of $C$, except possibly $P$ itself.
\end{enumerate}
That such $\alpha,\beta,\gamma$ exist is not hard to see: first choose $R$ on $\lambda$, but not on $C^\prime$, then consider all lines $\mu$ contained in $\pi$ through $R$, then the first condition is satisfied. The last three conditions can only prohibit finitely many choices of $\mu$, but since we now work over an algebraically closed field, we have infinitely many possibilities left.
The point of choosing the lines $\mu$ in this way is that the divisor of zeroes of $f$ in $F^\prime$, $(f)_{0}^{F^\prime}$, is equal to the intersection divisor of $C^\prime$ and $\mu$, identifying non-singular points of $C^\prime$ with places of $F^\prime$.
Hence $\deg((f)_{0}^{F^\prime})=d^\prime,$ which implies that $\deg((f)_{0}^{F})=\delta d^\prime$.
The zeroes of the function $f$, seen as element of the function field of $C$, are intersection points of $C$ and the plane $\sigma$ spanned by $\mu$ and $P$. The point $P$ itself is not accounted for, since none of the points in $C^\prime \setminus \rho_P(C \setminus\{P\})$ lies on the line $\mu$. This means that the degree of the intersection divisor of $\sigma$ with $C$ is at least $\deg((f)_{0}^{F^\prime})+i(C \cap \sigma,P)=\delta d^\prime+i(C \cap \sigma,P)$.
The result follows by observing that if $P$ is non-singular, then $i(C \cap \sigma,P) \ge 1$, while if $P$ is singular, then $i(C \cap \sigma,P)\ge 2$.
\end{proof}

\begin{lemma} \label{lem3}
Let $C$ be an absolutely irreducible, non-degenerate curve in $\mathbb{P}^3$ defined over $\mathbb{F}_q$ of degree $d$.
Assume that there exists a point $P \in C(\mathbb{F}_q)$ such that every $\mathbb{F}_q$-rational line $\ell$ containing $P$ satisfies $|\ell \cap C(\mathbb{F}_q)| \leq 2$. Moreover assume that if $P$ is a non-singular point, then the tangent line of $C$ at $P$ intersects $C$ in no other rational points. Then $|C(\mathbb{F}_q)| \le q(d-2)+1$. Equality may only occur if $P$ is a non-singular point of $C$.
\end{lemma}

\begin{proof}
Choose $\pi$ a plane not containing $P$. We consider the rational map $\rho_P: C \to C^\prime$ coming from the projection of $C$ from $P$ on $\pi$. As in Lemma \ref{lem:projection_map}, we denote by $d^\prime$ the degree of $C^\prime$ and by $\delta$ the degree of the rational map $\rho_P$.
First of all, note that any point of $C(\mathbb{F}_q)$ different from $P$ is projected to a point of $C^\prime(\mathbb{F}_q)$. Indeed, if $Q \in C(\mathbb{F}_q)$, and $P \neq Q$, the line connecting $P$ and $Q$ is rational and hence intersects $\pi$ in a rational point of $C^\prime$. Since by assumption any line through $P$ contains at most one further $\mathbb{F}_q$-rational point of $C$, the rational points of $C$ distinct from $P$ are mapped injectively to the rational points of $C^\prime$.
If moreover $P$ is non-singular, the projection map can be extended to all of $C$ and $P$ is then projected to a rational point $\rho_P(P)$ of $C^\prime$ as well, namely the intersection point of the tangent line of $C$ at $P$ and the plane $\pi$. Moreover, since the tangent line at $P$ does not contain any other rational points of $C$, $P$ is the only point projected to $\rho_P(P).$
Combining the above, we see that:
\begin{equation}\label{eq:boundCCprime}
|C(\mathbb{F}_q)| \le
\begin{cases}
|C^\prime(\mathbb{F}_q)|, & \text{if $P$ is non-singular,}  \\
|C^\prime(\mathbb{F}_q)|+1, & \text{if $P$ is singular.} \end{cases}
\end{equation}
At this point, we use Theorem \ref{plane}, distinguishing two cases:
\begin{enumerate}
\item[(i)] $C^\prime$ is not isomorphic as a plane curve  over $\mathbb{F}_q$ to the exceptional curve $K$ described in Theorem \ref{plane}. From Lemma \ref{lem:projection_map}, we see that $d^\prime \le d-1$ and in fact $d^\prime \le d-2$ if $P$ is singular.
    Therefore equation \eqref{eq:boundCCprime} yields
    \begin{equation*}
|C(\mathbb{F}_q)| \le
\begin{cases}
q(d-2)+1, & \text{if $P$ is non-singular,}  \\
q(d-3)+2, & \text{if $P$ is singular.} \end{cases}
\end{equation*}
Note that $q(d-3)+2 \le q(d-2)$, so the bound $q(d-2)+1$ may only be attained if $P$ is a non-singular point of $C$.
\item[(ii)] $q=4$ and $C^\prime$ is isomorphic as a plane curve over $\mathbb{F}_4$ to the exceptional curve $K$. Since $K$ has degree four, we see from Lemma \ref{lem:projection_map} that $d \ge 5$ and in fact $d \ge 6$ if $P$ is a singular point of $C$ or if $\delta>1$. If $d \ge 6$, the same reasoning as above gives that
    $|C(\mathbb{F}_4)| \le |C^\prime(\mathbb{F}_4)|+1 = 15 < 4(d-2)+1$.
    Hence we are left with the case where $\delta=1$, $d=5$, and $P$ is non-singular. Note that the curve $K$ is a non-singular quartic and therefore has genus three. Hence $g(C) \ge g(C^\prime)=3$. Theorem \ref{castelnuovo} then implies that the curve $C$ is strange. Note that by Theorem \ref{strange} the curve $K$, and hence $C^\prime$, is not strange.

    Denote by $N$ the nucleus of $C$. If $N \neq P$, then denote by $\rho_P(N)$ the projection of $N$ on $\pi$. Note that the line $t_P$ connecting $P$ and $N$ is the tangent line of $C$ at $P$, so that $\rho_P(P)=\rho_P(N) \in C^\prime$.
    Let $Q \in C$ be different from $P$ and $N$. If $Q$ is a non-singular point of $C$, denote by $t_Q$ be the tangent line of $C$ at $Q$. Then $\rho_P(t_Q)$ is the line passing through $\rho_P(Q)$ and $\rho_P(N)$ and since $\delta=1$, it is the tangent line of $C^\prime$ at $\rho_P(Q)$. Moreover, again using that $\delta=1$, any non-singular point of $C^\prime$ is the projection of a non-singular point of $C$. This shows that $\rho_P(N)$ is a nucleus of $C^\prime$, which therefore is a strange curve. However, $C^\prime$ is not a strange curve, giving a contradiction.

    If $N=P$ and let $\sigma$ be a plane containing the tangent line $t_P$ of $C$ at $P$. For every other point $Q \ne P$ in $C(\mathbb{F}_4) \cap \sigma$, the intersection multiplicity of $C$ and $\sigma$ at $Q$ is at least $2$, since either $Q$ is a singular point, or $\sigma$ contains the tangent line of $C$ at $Q$, which is the line passing through $P$ and $Q$. Since $d=5$, B{\'e}zout's theorem implies that $|C(\mathbb{F}_4) \cap \sigma| \leq 2$. Considering the set $\mathcal{B}(t_P)$ of planes in $\mathbb{P}^3(\mathbb{F}_4)$ containing $t_P$, we get that $|C(\mathbb{F}_4)| \leq 1+|\mathcal{B}(t_P)|(2-1)=1+5=6$.
\end{enumerate}
\end{proof}

So far we have in Lemmas \ref{lem2} and \ref{lem3} dealt with cases where rational lines intersect $C$ in at least four or at most two rational points.
This leaves the case where for every point $P \in C(\mathbb{F}_q)$ no line through $P$ intersects $C$ in four or more rational points, but there exists a line through $P$ intersecting $C$ in three rational points. 
We deal with this case in the following lemmas.

\begin{lemma} \label{lem3pts}
Let $C$ be an absolutely irreducible, non-degenerate curve in $\mathbb{P}^3$ defined over $\mathbb{F}_q$ of degree $d$.
Assume that there exists a point $P \in C(\mathbb{F}_q)$ such that every $\mathbb{F}_q$-rational line $\ell$ containing $P$ satisfies $|\ell \cap C(\mathbb{F}_q)| \leq 3$. 
Further let $\pi$ be an $\mathbb{F}_q$-rational plane not containing $P$ and denote by $C^\prime$ the projection of $C$ on $\pi$. If the degree of the induced rational map $\rho_P: C \to C^\prime$ is at least two, then $|C(\mathbb{F}_q)| \le q(d-2)$.
\end{lemma}
\begin{proof}
 As before, denote by $d^\prime$ the degree of the curve $C^\prime$ and by $\delta$ the degree of $\rho_P$.
 By Remark \ref{rem:singular_intersection}, we may assume that if $P$ is non-singular, then the tangent line of $C$ at $P$ intersects $C$ in at most one further rational point of $C$. Since any other line through $P$ contains at most $2$ rational points of $C \setminus \{P\}$, we see that
\begin{equation*}
|C(\mathbb{F}_q)| \le
\begin{cases}
2|C^\prime(\mathbb{F}_q)|, & \text{if $P$ is non-singular,}  \\
2|C^\prime(\mathbb{F}_q)|+1, & \text{if $P$ is singular.} \end{cases}
\end{equation*}
The $1$ on the right-hand side of the second case account for the contribution of the point $P$.

This implies that if $C^\prime$ is not isomorphic to the exceptional curve $K$, then using Lemma \ref{lem:projection_map}, we find that
$$|C(\mathbb{F}_q)| \le 2(q(d^\prime-1)+1) \le q\delta d^\prime -2q+2 \le q(d-1)-2q+2\le q(d-2),$$
if $P$ is non-singular, and
$$|C(\mathbb{F}_q)| \le 2(q(d^\prime-1)+1)+1 \le q\delta d^\prime -2q+3 \le q(d-2)-2q+3< q(d-2),$$
if $P$ is singular.

If $C^\prime$ is isomorphic to the exceptional curve $K$, then $d^\prime=4$ so that $d \ge 9$ if $P$ is non-singular and $d \ge 10$ if $P$ is singular by Lemma \ref{lem:projection_map}.
Therefore
$|C(\mathbb{F}_4)| \le 2\cdot14 = 4(9-2)$, if $P$ is non-singular, and $|C(\mathbb{F}_4)| \le 2\cdot14+1 < 4(10-2)$, if $P$ is singular.
\end{proof}

Finally we deal with the case that the rational map $\rho_P: C \to C^\prime$ has degree one.

\begin{lemma} \label{lem4}
Let $C$ be an absolutely irreducible, non-degenerate curve in $\mathbb{P}^3$ defined over $\mathbb{F}_q$ of degree $d$.
Assume that there exists a point $P \in C(\mathbb{F}_q)$ such that every $\mathbb{F}_q$-rational line $\ell$ containing $P$ satisfies $|\ell \cap C(\mathbb{F}_q)| \leq 3$. 
Further let $\pi$ be an $\mathbb{F}_q$-rational plane not containing $P$ and denote by $C^\prime$ the projection of $C$ on $\pi$. If the induced rational map $\rho_P: C \to C^\prime$ is birational, then $|C(\mathbb{F}_q)| \le q(d-2)-1$.
\end{lemma}

\begin{proof}
As in the proof of the previous lemma, we may assume that if $P$ is non-singular, then the tangent line $t_P$ of $C$ at $P$ intersects $C$ in at most one further rational point of $C$.
Let $L$ be the set of lines in $\mathbb{P}^3(\mathbb{F}_q)$ passing through $P$ and let
$$L_3(P):=\{\ell \in L \mid |\ell \cap C(\mathbb{F}_q)|=3 \text{ or } \ell=t_P \text{ and } |\ell \cap C(\mathbb{F}_q)|=2\}.$$
By Lemma \ref{lem3}, the bound holds if $\ell_3(P):=|L_P(3)|=0$, so we can assume that $\ell_3(P)> 0$.

Also $C^\prime=\pi_P(C)$ is absolutely irreducible of degree at most $d-1$. Since $\rho_P$ is a birational map, for any $\ell \in L_3(P)$ the corresponding intersection point $R$ with the projection plane $\pi$ is a singular point. Indeed, the preimage of $\rho_P^{-1}(R)$ consists of at least two points in $C$, implying that the function field of $C^\prime$, which is isomorphic to that of $C$, contains at least two distinct places corresponding to $R$.

Fix such a singular point $R$ coming from $\ell \in L_3(P)$ and let $L^\prime$ be the set of lines in $\pi$ passing through $R$. Since the degree of $C^\prime$ is at most $d-1$ and $R$ is a singular point of $C^\prime$, each of the $q+1$ lines $\ell^\prime \in L^\prime$ satisfies $|C^\prime(\mathbb{F}_q)\cap \ell^\prime\setminus\{R\}| \le d-3$. This gives the bound $|C^\prime(\mathbb{F}_q)| \le (q+1)(d-3)+1$. This bound can be improved however. Apart from $R$ are at least $\ell_3(P)-1$ further singular points in $C^\prime(\mathbb{F}_q)$, one for each line in $L_3(P)\setminus \{\ell\}$. Each of these singular points will occur with multiplicity at least two on one of the lines of $L^\prime$. Therefore, we obtain $$|C^\prime(\mathbb{F}_q)| \le (q+1)(d-3)-\ell_3(P)+2.$$
This enables us to estimate the number of rational points on $C$. Each rational point of $C^\prime$ has at most one preimage in $C$ under $\rho_P$, unless it is one of the $\ell_3(P)$ singularities coming from the lines in $L_3(P)$ in which case the preimage consists of two points. Therefore, we obtain that
$$|C(\mathbb{F}_q)| \leq (d-3)(q+1)-\ell_3(P)+2+\ell_3(P)=(d-2)q-q+1 \le (d-2)q-1.$$
\end{proof}

We have now shown that Conjecture \ref{conj} is true for $n=3$ and hence that Theorem \ref{main} is true as well.

\begin{theorem} \label{cor1}
Let $\mathbb{F}_q$ be a finite field with $q$ elements. Let $C$ be an absolutely irreducible, non-degenerate algebraic curve in $\mathbb{P}^3$ defined over $\mathbb{F}_q$ of degree $d$.
Then $|C(\mathbb{F}_q)| \leq q(d-2)+1$.
\end{theorem}

\begin{proof}
If $d \geq 2q+1$, then the claim follows from Lemma \ref{lem1}. So we can assume that $d \leq 2q$.
If there exists an $\mathbb{F}_q$-rational line $\ell$ such that $|\ell \cap C(\overline{\mathbb{F}}_q)|\geq 4$, then the result follows from Lemma \ref{lem2}.
Likewise, if there exists a $\mathbb{F}_q$-rational line $\ell$ such that $|\ell \cap C(\overline{\mathbb{F}}_q)|= 3$, and $\ell$ is a tangent line of $C$, then the result follows from Remark \ref{rem:singular_intersection}.

Now assume that there exists a point $P \in C(\mathbb{F}_q)$ such that every $\mathbb{F}_q$-rational line $\ell$ containing $P$ satisfies $|\ell \cap C(\mathbb{F}_q)| \leq 2$. Moreover, assume that if $P$ is a non-singular point, then the tangent line of $C$ at $P$ intersects $C$ in no other rational points. In this setting, the result follows from Lemma \ref{lem3}.

This leaves the case in which for every point $P \in C(\mathbb{F}_q)$ no line through $P$ intersects $C$ in four or more rational points, but there exists a line through $P$ intersecting $C$ in three rational points.
If there exists an $\mathbb{F}_q$-rational point $P \in C$ and an $\mathbb{F}_q$-rational plane $\pi$ not containing $P$ such that the degree of the corresponding rational map $\rho_P: C \to C^\prime$ is at least two, then the result follows from Lemma ref{lem3pts}, while if $\rho_P$ is birational, then the claim follows from Lemma \ref{lem4}.
\end{proof}

\section{Some extremal examples}

In this section we provide some families of extremal examples with respect to the bounds proven in the previous section. Looking at the proof of Theorem \ref{cor1}, the bound $q(d-2)+1$ can only be attained if $d \leq 2q$, all $\mathbb{F}_q$-rational points $P \in C$ are non-singular and for all $\mathbb{F}_q$-rational lines $\ell$ containing $P$ one has $|\ell \cap C(\mathbb{F}_q)| \leq 2$. This already puts several restrictions on finding extremal examples and it seems that curves attaining equality in Conjecture \ref{conj} are not very common.
%
\begin{example}
The rational normal curve in is a non-singular, rational curve $\mathcal{R}_n$ of degree $n$ in the $n$-dimensional projective space $\mathbb{P}^n(\mathbb{F}_q)$, with $n \geq 1$. It is the Veronese variety when the domain is the projective line. For $n = 2$ it is the plane conic $xy = z^2$ and for $n = 3$ it is the twisted cubic. A model for $\mathcal{R}_3$ can be constructed as the intersection locus of three non-singular quadrics,
$$\mathcal{R}_3: \begin{cases} xz-y^2=0, \\ yw-z^2=0, \\ xw-yz=0. \end{cases}$$
The curve $\mathcal{R}_n$ has degree $n$ and $q+1=q(n-n+1)+1=q+1$ $\mathbb{F}_q$-rational points.
Hence the bound in Conjecture \ref{conj} is attained by these curves.
It is in fact well known that any absolutely irreducible, non-degenerate algebraic curve in $\mathbb{P}^n$ of degree $n$ defined over $\mathbb{F}_q$ is a rational curve, see \cite[Exercise 3.4]{Hart}. Hence if $d=n$, Conjecture \ref{conj} is valid.
\end{example}

\begin{example}
For $q \leq 4$ explicit, plane, non-singular elliptic curves having $2q+1$ $\mathbb{F}_q$-rational points exist. These attain the bound in Theorem \ref{plane}. Examples are the curve $\epsilon_2: y^2z+yz^2=x^3+xz^2$ for $q=2$, $\epsilon_3: y^2z=x^3+2xz^2+z^3$ for $q=3$ and $\epsilon_4: y^2z+yz^2=x^3$ for $q=4$. Using a basis of the Riemann--Roch space $L((n+1)P_\infty)$, which has dimension $n+1$ if $n \ge 3$, one can embed these elliptic curves in $\mathbb{P}^n$ as non-degenerate, degree $n+1$ curves. Then $(d-n+1)q+1=2q+1$, so the bound in Corollary \ref{conj} is attained.
\end{example}

\begin{example}
According to Lemma \ref{lem1}, an absolutely irreducible, non-degenerate space curve defined over $\mathbb{F}_2$ of degree five can have at most six rational points. The hyperelliptic plane curve $D$ given by the equation $y^2z^3+(x^3+xz^2+z^3)yz=x^5+x^4z+x^3z^2+xz^4$ (source: www.manypoints.org) is an absolutely irreducible curve of genus two with five rational points. The point $[0:1:0]$ is a singular point, corresponding to two rational places $P_\infty^{(1)}$ and $P_\infty^{(2)}$ in the function field of $C$. Let $\{1,x_1,x_2,x_3\}$ be a basis of the Riemann--Roch space $L(5P_\infty^{(1)}).$ Now define $C$ to be the Zariski closure of the image of the map $\iota: D\setminus\{P_\infty^{(1)},P_\infty^{(2)}\} \to \mathbb{P}^3$, defined by $\iota(P)=[1:x_1(P):x_2(P):x_3(P)]$. The curve C is non-degenerate, has degree five and six $\mathbb{F}_2$-rational points, attaining the bound from Lemma \ref{lem1}.
\end{example}

We finish by making some comment on the hypotheses in Theorem \ref{main}.
First of all, the hypothesis that $C$ does not contain an $\mathbb{F}_q$-rational line cannot be removed. Indeed, let $C$ be the union of the twisted cubic in $\mathbb{P}^3$ and an $\mathbb{F}_q$-rational line not intersecting the twisted cubic. Then $C$ has degree $d=4$, so that $q(d-n+1)+1=2q+1$, while $|C(\mathbb{F}_q)|=2q+2$.
Likewise, the hypothesis that $C$ has at least one non-degenerate $\mathbb{F}_q$-irreducible component cannot be removed. Consider for example the elliptic curve $\mathcal{E} \subset \mathbb{P}^2$ defined over $\mathbb{F}_3$ by the equation $y^2t=x^3-xt^2+t^3$.
Then $|\mathcal{E}(\mathbb{F}_3)|=7$. The line $\ell$ defined by the equation $y=0$, intersects the set $\mathcal{E}(\mathbb{F}_3)$ trivially. This configuration can be embedded in $\mathbb{P}^3$ by using the coordinate system $[x:y:z:t]$ and mapping a point $[x:y:t] \rightarrow [x:y:0:t]$ and hence as a degenerate curve $\mathcal{E}$ is contained in the plane $\pi$ given by the equation $z=0$. A rotation $\tau_\ell$ defined over $\mathbb{F}_3$ of the plane $\pi$ fixing the line $\ell$ will map $\mathcal{E}$ to another degenerate elliptic curve $\tau_\ell(\mathcal{E})$ with $|\tau_\ell(\mathcal{E})(\mathbb{F}_3)|=7$ and $\tau_\ell(\mathcal{E})(\mathbb{F}_3) \cap \mathcal{E}(\mathbb{F}_3)=\emptyset$. Hence the $\mathbb{F}_3$-reducible, non-degenerate curve $C=\mathcal{E} \cup \tau_\ell(\mathcal{E})$ will be of degree $6$ and $|C(\mathbb{F}_3)|=14>q(d-2)+1=3(6-2)+1=13$.

\section*{Acknowledgements}

The first author would like to acknowledge the support from The Danish Council for Independent Research (DFF-FNU) for the project \emph{Correcting on a Curve}, Grant No.~8021-00030B. The first author would also like to thank Prof. Masaaki Homma for several stimulating research discussions, whose inspiring effects led to the formulation of Conjecture \ref{conj}.


\begin{thebibliography}{99}
\bibitem{Ballico} E. Ballico, Special inner projections of projective varieties, Ann. Univ. Ferrara \textbf{50}, 23-–26 (2004).
\bibitem{BH} V. Bayer and A. Hefez, \textit{Strange curves}, Comm. Algebra \textbf{19}, 3041--3059 (1991).
\bibitem{Hart} R. Hartshorne, \textit{Algebraic geometry}, Graduate Texts in Mathematics \textbf{52}. Springer-Verlag, New York-Heidelberg, (1977).
\bibitem{HKT} J.W.P.~Hirschfeld, G.~Korchm\'aros and F.~Torres, \emph{Algebraic Curves over a Finite Field,} \textit{Princeton Series in Applied Mathematics}, Princeton, (2008).

\bibitem{Homma} M. Homma, A bound on the number of points of a curve in projective space over a finite field, Theory and Applications of Finite Fields in: \emph{Contemp. Mat.} \textbf{579}, 103--110 (2012).

\bibitem{H1} M. Homma and S.J. Kim, Around Sziklai's conjecture on the number of points of a plane curve over a finite field, \emph{Finite Fields and Appl.} \textbf{15}, 468--474 (2009).

\bibitem{H2} M. Homma and S.J. Kim, Sziklai's conjecture on the number of points of a plane curve over a finite field II, in: G. McGuire, G.L. Mullen, D. Panario, I.E. Shparlinski (Eds.), Finite Fields: Theory and Applications, in: \emph{Contemp. Math.} \textbf{518}, AMS, Providence, 225--234 (2010).

\bibitem{H3} M. Homma and S.J. Kim, Sziklai's conjecture on the number of points of a plane curve over a finite field III, \emph{Finite Fields and Appl.} \textbf{16}, 315--319 (2010).


\bibitem{Samuel} P. Samuel, \textit{Lectures on old and new results on algebraic curves (notes by Anantharaman)}, Tata Inst. Fund. Res., (1966).
\bibitem{PSzi} P. Sziklai, A bound on the number of points of plane curves, \emph{Finite Fields and Appl.} \textbf{14}, 41--43 (2008).
\bibitem{SV} K.O. St\"ohr, J.F. Voloch, \textit{Weierstrass points and curves over finite fields}, Proc. London Math. Soc. \textbf{52}, 1--19 (1986).
\end{thebibliography}
    \end{document}